\documentclass[a4paper]{amsart}

\usepackage{amsthm, amssymb, amsmath, amsfonts}
\usepackage{url}

\newtheorem{thm}{Theorem}[section]
\newtheorem{prop}[thm]{Proposition}

\renewcommand{\le}{\leqslant}
\renewcommand{\ge}{\geqslant}

\newcommand{\E}{\mathbb{E}}
\newcommand{\EE}{\mathbf{E}}
\newcommand{\cE}{\mathcal{E}}
\newcommand{\N}{\mathbb{N}}
\renewcommand{\L}{\mathcal{L}}
\newcommand{\1}{\mathbf{1}}
\newcommand{\R}{\mathbb{R}}
\newcommand{\Z}{\mathbb{Z}}
\renewcommand{\P}{\mathbb{P}}
\newcommand{\PP}{\mathbf{P}}
\newcommand{\ov}{\overline}
\newcommand{\td}{\tilde}
\newcommand{\eps}{\varepsilon}
\def\d{{\mathrm{d}}}

\title{Principal eigenvalue for random walk among random traps on $\Z^d$}
\author{Jean-Christophe Mourrat}

\address{Universit\'e de Provence, CMI, 39 rue Joliot Curie, 13013 Marseille, France ; PUC de Chile, Facultad de Matem\'aticas, Vicu\~na Mackenna 4860, Macul, Santiago, Chile.
}

\begin{document}

\begin{abstract}
Let $(\tau_x)_{x \in \Z^d}$ be i.i.d. random variables with heavy (polynomial) tails. Given $a \in [0,1]$, we consider the Markov process defined by the jump rates $\omega_{x \to y} = {\tau_x}^{-(1-a)} {\tau_y}^a$ between two neighbours $x$ and $y$ in $\Z^d$. We give the asymptotic behaviour of the principal eigenvalue of the generator of this process, with Dirichlet boundary condition. The prominent feature is a phase transition that occurs at some threshold depending on the dimension. 
\end{abstract}

\maketitle

\section{Introduction}
For each site $x \in \Z^d$, let $\tau_x > 0$ be a random variable, so that $(\tau_x)_{x \in \Z^d}$ are independent and identically distributed. We call $\tau = (\tau_x)_{x \in \Z^d}$ the \emph{environment}, and write its law $\P$ (and the corresponding expectation $\E$).
Fixing $a \in [0,1]$ and an environment $\tau$, we define the Markov process $(X_t)_{t \ge 0}$ by the following jump rates~:
$$
\omega_{x \rightarrow y} = \left|
\begin{array}{ll}
{\tau_x}^{-(1-a)} {\tau_y}^a & \text{if } \|x-y\|=1 ,\\
0 & \text{otherwise}.
\end{array}
\right.
$$
The associated infinitesimal generator is~:
$$
\L f(x) = \sum_{y : \|x-y\|=1} \omega_{x \rightarrow y} (f(y)-f(x)).
$$
The aim of this note is to investigate the behaviour of the principal eigenvalue of $\L$ restricted to a large box. Define the box of size $n$ by $B_n = \{-n, \ldots, n\}^d$, and $\L_n$ the operator $\L$ restricted to this box, with Dirichlet boundary conditions. That is to say $\L_n f = \1_{B_n} \L f$, defined for any function $f : \Z^d \to \R$ that vanishes outside the box. Let $\lambda_n$ be the smallest eigenvalue of $-\L_n$. We write $\lambda_n^\circ$ for the eigenvalue obtained in the particular case when $a=0$.

We are particularly interested in the study of heavy tailed laws for the environment. A natural assumption (see the remark just after Theorem \ref{main}) is that the tail $\P[\tau_0 > y]$ behaves like a power of $y$ as $y$ goes to infinity. 

\noindent \textbf{Assumption 1.} There exists $\alpha > 0$ such that~:
\begin{equation}
\label{assump1}
F(y) := \P[\tau_0 > y] \simeq \frac{1}{y^\alpha} \qquad (y \rightarrow + \infty).
\end{equation}

More precisely, we say that a function $f$ varies regularly with index $\rho$ at infinity, and write $f \in RV_\rho$, if for all $\kappa > 0$, $f(\kappa x)/f(x) \to \kappa^\rho$ as $x \to + \infty$ (see \cite{reg} for a monograph on regular variation). 

\noindent \textbf{Assumption 1'.} There exists $\alpha > 0$ such that $F \in RV_{-\alpha}$.

Assumption 1' gives a precise sense to assumption 1, and is more general than just assuming the equality (or equivalence) in equation (\ref{assump1}). Note that, for $0 < \alpha < 2$, $\tau_0$ belongs to the domain of attraction of an $\alpha$-stable law if and only if $F \in RV_{-\alpha}$ (see \cite[Corollary XVII.5.2]{fel}). 
Assumption 1' implies that for any $\eps > 0$~:
\begin{equation}
\label{ppF}
F(y) y^{\alpha+\eps} \xrightarrow[y \to +\infty]{} + \infty \quad \text{and} \quad F(y) y^{\alpha-\eps} \xrightarrow[y \to +\infty]{} 0,
\end{equation}
and as a consequence, $\E[\tau_0^\beta]$ is finite for all $\beta < \alpha$, infinite for all $\beta > \alpha$ (and may be finite or infinite when $\beta = \alpha$).

\noindent {\bf Assumption 2.} We will always assume that $\tau_0 \ge 1$, concentrating on ``bad behaviours'' at infinity. 
 
We need to introduce the generalized inverse of $1/F$, defined by~:
$$
h(x) = \inf \{y : 1/F(y) \ge x\}.
$$
As $F$ belongs to $RV_{-\alpha}$, one can see that $h \in RV_{1/\alpha}$ (see for instance \cite[Proposition 0.8 (v)]{res}). Loosely speaking, $h(y) \simeq y^{1/\alpha}$. We will recall later how $h$ is related to the asymptotic behaviour of maxima and sums of $(\tau_x)$ (see Proposition \ref{hbehav}), but let us first state (and comment) our main results.


\begin{thm}
\label{main}
\begin{enumerate}
\item
\label{main:psa}
For almost every environment, we have~:
$$
\lim_{n \to \infty} - \frac{\ln(\lambda_n)}{\ln(n)} = 
\left|
\begin{array}{ll}
\max\left(2,1+\dfrac{1}{\alpha}\right) & \text{if } d = 1 ,\\
\max\left(2,\dfrac{d}{\alpha}\right) & \text{if } d \ge 2.
\end{array}
\right.
$$
\item
\label{main:psr}
If $d \ge 2$ and $\alpha > d/2$ or if $d=1$ and $\alpha > 1$, then there exist $k_1,k_2 > 0$ such that for almost every environment and $n$ large enough~:
$$
\frac{k_1}{n^2} \le \lambda_n \le \frac{k_2}{n^2}.
$$
\item
\label{main:dn2}
If $\alpha < 1$ and $d \neq 2$, then for any $\eps > 0$, there exist $\eta, M > 0$ such that for all $n$ large enough~:
$$
\P[\eta \le a_n \lambda_n \le M] \ge 1 - \eps,
$$
where
$$
a_n =
\left|
\begin{array}{ll}
n h(n) & \text{if } d=1 ,\\
h(n^d) & \text{if } d \ge 3.
\end{array}
\right.
$$
\item
\label{main:d2}
Let $a_n = \ln(n) h(n^2)$. If $d=2$ and $\alpha < 1$, then for any $\eps > 0$, there exist $\eta, M > 0$ such that for all $n$ large enough~:
$$
\P[\eta \le a_n \lambda_n^\circ \le M] \ge 1 - \eps,
$$
$$
\P[\eta \le a_n \lambda_n \le \ln(n) M] \ge 1 - \eps.
$$
\end{enumerate}
\end{thm}

Let us now give some heuristics about the behaviour of $(X_t)$. If $a=0$, the walk is in fact a time-change of the simple random walk~: arriving at some site $x$, it waits an exponential time of mean $\tau_x$ before jumping to a neighbouring site chosen uniformly. When $a \neq 0$, things get more complicated. Suppose that the walk arrives at some deep trap, that is a site $x$ where $\tau_x$ is very large. Compared with the $a=0$ case, the walk will leave site $x$ faster. On the other hand, once on a neighbouring site, it will come back to $x$ with very high probability. These two competing effects can compensate remarkably in the limit, and indeed our main results are independent of $a$ (as they also are in \cite{dim1}).

We propose to call $(X_t)_{t \ge 0}$ a \emph{random walk among random traps}. It seems to us that for its relative simplicity, it should be considered one of the basic types of random walks in random environments to study, just as is the random walk among random conductances. Although one could have the feeling that theses two types are basically the same, one attaching randomness to edges of the graph and the other to sites, they exhibit very different behaviours. For instance, the reversible measure is not the uniform one in the case of random traps (it gives weight $\tau_x$ to site $x$). Also, if $d \ge 2$, the random walk in random conductances tends to avoid visiting regions where conductance is very low (and where time spent to `get out' may be high). On the other hand, when walking among random traps, say for $a = 0$, the path is the same as for the simple random walk, and the walk is not inclined to avoid regions from which it takes a long time to get out. See \cite{alex} for a nice discussion about this issue.

This type of walk gained interest when J.P.~Bouchaud \cite{bou} proposed it as a phenomenological model to explain aging of spin glasses, and as a consequence, what we call `random walk among random traps' is also known as \emph{Bouchaud's model}. Later on, \cite{rmb} introduced the full model as presented here (including the $a \in [0,1]$), which allows them to get more diverse aging scalings.

When $\E[\tau_0]$ is finite (in particular when $\alpha > 1$), one can apply results of \cite{masi} to prove that, under the averaged law, $(X_t)$ is diffusive and converges to Brownian motion after rescaling.

In one dimension, for $a=0$ and $\alpha < 1$, L.R.G.~Fontes, M.~Isopi and C.M.~Newman \cite{fin} proved that almost surely the process was subdiffusive and obtained convergence of the rescaled process to a singular diffusion, as well as aging. The results have been extended to general $a$ by G.~Ben Arous and J.~\v{C}ern\'y in \cite{dim1}. Another (also subdiffusive) scaling limit was identified when $d \ge 2$, $\alpha < 1$ and $a=0$ in \cite{dim2}. We refer to \cite{rev} for a review on the subject. To our knowledge, nothing was known in the case when $a \neq 0$ and $d \ge 2$. 

This note comes as a partial answer to a question of \cite{rev}, asking for the ``nature of the spectrum of the Markov chain close to its edge. Naturally, the long time behaviour of $X_t$ can be understood from the edge of the spectrum of the generator $\L$. This question deserves further study (see \cite{bovfag1}, \cite{bovfag2} and also \cite{mb97}).''

Upper bounds on $\lambda_n$ are obtained rather easily, using its variational characterisation (see equation~(\ref{variational})), and then choosing appropriate test functions. 
Finding the corresponding lower bounds is more difficult. Remarkably, the classical techniques exposed for instance in the review \cite{SC}, although giving the appropriate bounds in certain cases, did not enable us to conclude in general. We show in section~\ref{s:cs} that the distinguished path method (see e.g. \cite[Theorem 3.2.3]{SC}), that proved efficient for instance in \cite[Section 3]{fm} for random walks among random conductances, is bound to give an extra $1$ in the exponent when $d \ge 2$ (for the one-dimensional case, \cite[Section 3.7]{chen} proves that the method is sharp, as can be checked directly in our context). In order to solve the problem, we have chosen to bound the exit times of the walk from $B_n$. For $a = 0$ and $d \ge 3$, we find estimates on these exit times using the knowledge of the Green function of the embedded discrete-time simple random walk, together with a moments computation. This method can be modified to treat as well the two-dimensinal case, and an elementary argument extends the bounds to general $a$, see equation~(\ref{lambdalambda0}) (but note that it can in fact be applied directly to general $a$, see a previous version of this note \cite{version1}). However, we would like to draw reader's attention to the fact that this method gives little indication on how to extend the results to a conservative dynamics (for instance, with periodic boundary conditions instead of Dirichlet).

\noindent \textbf{Remark.} A natural choice of $(\tau_x)$ from the statistical physics' point of view is the following~: first choose independently for each site a random variable $-E_x$ with law exponential of parameter $1$, 
and define $\tau_x$ to be $\exp(-\beta E_x)$, where $\beta$ represents the inverse of the temperature. Then one can check that $F \in RV_{-1/\beta}$, and the irregularity that appears at $\beta = 1$ for $d \le 2$ and at $\beta = 2/d$ for larger $d$ can be regarded as a phase transition (the anomalous behaviour occurring for $\beta$ large, that is for small temperature, or in our context, small $\alpha$).

It may seem surprising that this new phase transition does not appear at the same threshold than the diffusive/subdiffusive transition (that at least for $a=0$ occurs when $\alpha (=1/\beta) =1$ in any dimension). The reason for this is the following~: although the principal eigenvalue will `feel' the very deepest traps of the box (of order $n^{d/\alpha}$), the process (at least when $a=0$) will exit the box after visiting only some $n^2$ sites, thus having seen only traps of order at most $n^{2/\alpha}$.

Before going on to show how Theorem \ref{main} is a consequence of the results of the following sections, we need to recall some facts about the asymptotic behaviour of sums and maxima of $(\tau_x)$.
\begin{prop}
\label{hbehav}
\begin{enumerate}
\item
\label{hpetrov}
For any $\eps > 0$ and almost every environment~:
\begin{equation*}
n^{-(\max(d,d/\alpha) + \eps)} \sum_{x \in B_n} \tau_x \rightarrow 0 \qquad (n \to +\infty).
\end{equation*}
\item
\label{htriv}
For any $\eps > 0$ and almost every environment~:
\begin{equation*}
n^{-(\max(d,d/\alpha) - \eps)} \sum_{x \in B_n} \tau_x \rightarrow + \infty \qquad (n \to +\infty).
\end{equation*}
\item
\label{hmax}
There exists a random variable $M_{\infty}$ with values in $(0, + \infty)$ such that the re\-scaled maxima converge in law to $M_{\infty}$~:
\begin{equation*}
\frac{1}{h(n^d)} \ \max_{x \in B_n} \tau_x \Rightarrow M_{\infty} \qquad (n \to +\infty).
\end{equation*}
\item
\label{hsum}
If $\alpha < 1$, then there exists a random variable $S_{\infty}$ with values in $(0, + \infty)$ such that the rescaled partial sums converge in law to $S_{\infty}$~:
\begin{equation*}
\frac{1}{h(n^d)} \sum_{x \in B_n} \tau_x \Rightarrow S_{\infty} \qquad (n \to +\infty).
\end{equation*}
\end{enumerate}
\end{prop}
\begin{proof}
For the first statement, it is a consequence of the law of large numbers if $\alpha > 1$, otherwise it is an application of \cite[Theorem 6.9]{pet}. For the second one, it comes again from the law of large numbers if $\alpha > 1$. Otherwise, observe that the sum is larger than the maximum of its terms, and 
$$
\P\left[\max_{x \in B_n} \tau_x \le M n^{d/\alpha - \eps} \right] = (1-F(M n^{d/\alpha - \eps}))^{(2n+1)^d}.
$$
Using the properties of $F$ (see (\ref{ppF})), we see that the latter is the general term of a convergent series, and we can apply the Borel-Cantelli lemma.
Now the convergence of the rescaled maxima is given in \cite[Section VIII.8]{fel} or \cite[Proposition 1.11]{res}. For the convergence of the partial sums, see \cite[Section XVII.5]{fel}.
\end{proof}

Apart from this introduction, the paper is divided into four sections and an Appendix. In section \ref{s:lambda0}, we use the variational characterisation to get bounds on $\lambda_n^\circ$ and $\lambda_n$ that are sharp when $\alpha \le 1$ or $d=1$. In order to find a good lower bound on $\lambda_n^\circ$ (easily extended to a lower bound on $\lambda_n$) when $d \ge 2$ and $\alpha > 1$, we introduce in section \ref{s:Green} the embedded discrete time random walk. When $a=0$, it is the simple random walk, and the explicit knowledge of its Green function enables us to conclude. In section \ref{s:upper}, upper bounds for $\lambda_n$ are computed. Finally, we analyse the limitation of the distinguished path method in section \ref{s:cs}. 

Let us see how to deduce part (\ref{main:psa}) of Theorem \ref{main} from the rest of the paper. Part (\ref{anyalpha}) of Theorem \ref{thmlambda0} gives an upper bound on the exponent of the principal eigenvalue, that needs to be improved when $d \ge 3$ and $\alpha > 1$. This is done by Theorem \ref{green3}. Now for the associated lower bounds on the exponent of the principal eigenvalue, they come from Theorem \ref{upper1d} and part (\ref{htriv}) of Proposition \ref{hbehav} if $d=1$ ; from part (\ref{relaxps}) of Theorem \ref{relax} and Theorem \ref{taua} if $d \ge 2$.

Concerning part (\ref{main:psr}) of Theorem \ref{main}, if $d=1$ and $\alpha > 1$, the lower bound comes from part (\ref{finiteE}) of Theorem \ref{thmlambda0}. If $d\ge 2$  and $\alpha > d/2$, the lower bound is given by part (\ref{supn2}) of Theorem \ref{green3}. In any case, Theorem \ref{taua} gives the desired upper bound on $\lambda_n$.

Finally, for parts (\ref{main:dn2}) and (\ref{main:d2}) of Theorem \ref{main}, part (\ref{alpha<1}) of Theorem \ref{thmlambda0} gives the desired result for $\lambda_n^\circ$ as well as a lower bound on $\lambda_n$. In dimension one, the upper estimate on $\lambda_n$ is given by Theorem \ref{upper1d} and part (\ref{hsum}) of Proposition \ref{hbehav}, while if $d \ge 2$, it comes from part (\ref{relaxprob}) of Theorem \ref{relax} together with part (\ref{hmax}) of Proposition~\ref{hbehav}.

\noindent {\bf Notations.} 
The operator $\L_n$ is self-adjoint for the scalar product $(\cdot, \cdot)$ defined by~:
$$
(f,g) = \sum f(x)g(x) \tau_x.
$$ 
We write $L^2(B_n)$ for the set of functions that vanish outside $B_n$ (equipped with the former scalar product). For two points $x,y \in \Z^d$, we write $x \sim y$ when they are neighbours (that is, when $\|x-y\|=1$). We define the Dirichlet form associated to $\L$~:
\begin{eqnarray*}
\label{num}
\cE(f,g) = (-\L f,g) & = & \sum_{\substack{x,y \in \Z^d \\ x \sim y}} \tau_x^a \tau_y^a g(x)(f(x)-f(y)) \notag \\
 & = & \sum_{\substack{x,y \in \Z^d \\ x \sim y}} \tau_x^a \tau_y^a g(y)(f(y)-f(x)) \notag \\
 & = & \frac{1}{2} \sum_{\substack{x,y \in \Z^d \\ x \sim y}} \tau_x^a \tau_y^a (f(y)-f(x))(g(y)-g(x))
\end{eqnarray*}
(taking the half-sum of the last two expressions), and $\cE_0$ the Dirichlet form obtained when $a=0$. We have~:
\begin{equation}
\label{variational}
\lambda_n = \inf_{\substack{f \in L^2(B_n) \\ f \neq 0}} \frac{\cE(f,f)}{(f,f)}.
\end{equation}
Assumption 2 gives that $\cE(f,f) \ge \cE_0(f,f)$, so it is clear that 
\begin{equation}
\label{lambdalambda0}
\lambda_n \ge \lambda_n^\circ.
\end{equation}
We further need to define the boundary of $B_n$, as $\partial B_n = B_{n+1} \setminus B_n$. If $K$ is some set, $|K|$ stands for its cardinal. We write $\PP_x^\tau$ for the law of the process starting from site $x$ (and $\EE_x^\tau$ for the corresponding expectation).

The real number $C > 0$ represents a generic constant that need not be the same from one occurrence to another.

%
%
%
%
%
%
\section{The variational formula}
\label{s:lambda0}
\setcounter{equation}{0}
We will use here the variational characterisation of $\lambda_n^\circ$~:
\begin{equation}
\label{lambda0}
\lambda_n^\circ = \inf_{\substack{f \in L^2(B_n) \\ f \neq 0}} \frac{\cE_0(f,f)}{(f,f)}.
\end{equation}
We define 
$$
C_n = \inf \left\{\cE_0(f,f) \ | \ f \in L^2(B_n), f(0) = 1 \right\}.
$$
Noting that $B_n$ is a finite set, one can see by a compacity argument that the infimum is reached for some function $V_n$. The behaviours of $C_n$ and $\lambda_n^\circ$ are related in the following way.
\begin{prop}
\label{proplambda0}
For any $n$ and any environment, we have~:
$$
\frac{C_{2n}}{\sum_{x \in B_n} \tau_x} \le \lambda_n^\circ,
$$
$$
\lambda_{2n+1}^\circ \le \lambda_{2n}^\circ \le \frac{C_n}{\max_{B_n} \tau}.
$$
\end{prop}
\begin{proof}
Considering the homogeneity of the quotient in (\ref{lambda0}), we can restrict the infimum to be taken over all $f$ with $\|f\|_\infty = 1$. Let $f$ be such a function, and $x_0 \in B_n$ such that $|f(x_0)| = 1$. Possibly changing $f$ in $-f$, we can assume $f(x_0) = 1$. Noting that the function $g = f(\cdot + x_0)$ is in $L^2(B_{2n})$ and satisfies $g(0) = 1$, we have~:
$$
\cE_0(f,f) = \cE_0(g,g) \ge C_{2n}.
$$
On the other hand, as $\|f\|_\infty = 1$, we have~:
$$
(f,f) \le \sum_{x \in B_n} \tau_x,
$$
and these lead to the first desired inequality.

The fact that $\lambda_{2n+1}^\circ \le \lambda_{2n}^\circ$ is clear from (\ref{lambda0}). Now let $x_1 \in B_n$ be such that $\max_{B_n} \tau = \tau_{x_1}$, and consider the function $h = V_n(\cdot - x_1) \in L^2(B_{2n})$. We get~:
$$
\cE_0(h, h) = \cE_0(V_n,V_n) = C_n.
$$
But note that $h(x_1) = 1$, therefore~:
$$
(h,h) \ge \tau_{x_1} = \max_{B_n} \tau,
$$
and we get the second inequality.
\end{proof}
We now precise the asymptotic behaviour of $C_n$.
\begin{prop}
\label{Cn}
If $d=1$, then~:
$$
C_n = \frac{2}{n+1}.
$$
If $d=2$, then there exist $k_1, k_2$ such that for all $n$~:
$$
\frac{k_1}{\ln(n)} \le C_n \le \frac{k_2}{\ln(n)}.
$$
If $d \ge 3$, then $C_n$ converges to a strictly positive number.
\end{prop}
\begin{proof}
We can regard $B_{n+1}$ as an electrical graph (see \cite[Chapter 2]{trees}), with each edge representing a resistance of value $1$. One can see that $V_n$ is harmonic on every point that is not $0$ nor a point of $\partial B_n$. Thus it coincides with the potential on the electrical graph, with the constraints that $V_n(0) = 1$ and ${V_n}_{|\partial B_n} = 0$. The number $C_n$ is the effective conductance between $0$ and $\partial B_n$. In dimension 1, a direct computation gives the result. If $d=2$, then we can use \cite[Proposition 2.14]{trees}. In larger dimension, the simple random walk is transient, and therefore (see \cite[Theorem 2.3]{trees}) $C_n$ converges to a strictly positive number.
\end{proof}
From this, we can deduce the following.
\begin{thm}
\label{thmlambda0}
\begin{enumerate}
\item
\label{alpha<1}
If $\alpha < 1$, then for any $\eps > 0$, there exist $\eta, M > 0$ such that for all $n$ large enough~:
$$
\P\left[\eta \le  \frac{h(n^d)}{C_n} \lambda_n^\circ \le M\right] \ge 1 - \eps,
$$
$$
\P\left[\eta \le  \frac{h(n^d)}{C_n} \lambda_n\right] \ge 1 - \eps.
$$
\item
\label{anyalpha}
For almost every environment, we have~:
$$
\limsup_{n \to \infty} - \frac{\ln(\lambda_n)}{\ln(n)} \le 
\left|
\begin{array}{ll}
\max\left(2,1+\dfrac{1}{\alpha}\right) & \text{if } d = 1 ,\\
\max\left(d,\dfrac{d}{\alpha}\right) & \text{if } d \ge 2.
\end{array}
\right.
$$
\item 
\label{finiteE}
If $\E[\tau_0]$ is finite, then for almost every environment and all $n$ large enough~:
$$
\lambda_n \ge \frac{C_{2n}}{(2n+1)^d (\E[\tau_0]+1)}.
$$
\end{enumerate}
\end{thm}
\begin{proof}
Note first that as given by equation (\ref{lambdalambda0}), we have that $\lambda_n \ge \lambda_n^\circ$. The first part of the theorem is a consequence of Propositions \ref{proplambda0}, \ref{Cn} and parts (\ref{hmax}) and (\ref{hsum}) of Proposition \ref{hbehav}. For the second part, use part (\ref{hpetrov}) of Proposition \ref{hbehav} instead. The last part is an application of the law of large numbers.
\end{proof}
As far as lower bounds are concerned, parts (\ref{main:dn2}) and (\ref{main:d2}) of Theorem \ref{main} are now obtained. However, part (\ref{main:psa}) of Theorem \ref{main} is proved only for $d \le 2$ or $\alpha \le 1$, and part (\ref{main:psr}) only for $d=1$. The following section provides the missing lower bounds.
%
%
%
%
%
%
\section{Exit time upper bounds when $a=0$}
\label{s:Green}
\setcounter{equation}{0}
This section aims at finding good lower bounds for $\lambda_n$ when $d \ge 2$ and $\alpha > 1$. To do so, we will use the exit times $T_n$ from $B_n$~:
$$
T_n = \inf \{t \ge 0 : X_t \notin B_n \}.
$$
The principal eigenvalue and the exit time from $B_n$ are indeed related by the following (general) result~:
\begin{prop}
\label{sortie}
For any environment $\tau$, any $n \in \N$ and $t \ge 0$, we have
$$
e^{-t \lambda_n} \le \sup_{x \in B_n} \PP_x^\tau[T_n > t] \le \frac{\sup_{x \in B_n} \EE_x^\tau[T_n]}{t} .
$$
\end{prop}
\begin{proof}
Let $\psi_n$ be the eigenfunction associated with the principal eigenvalue $\lambda_n$ such that $\sup \psi_n = 1$.
$$
\EE_x^\tau[\psi_n(X_t) \1_{\{T_n > t\}}] = e^{-t \lambda_n} \psi_n(x).
$$
Choosing $x \in B_n$ such that $\psi_n(x) = 1$, we have~:
$$
\PP_x^\tau[T_n > t] \ge \EE_x^\tau[\psi_n(X_t) \1_{\{T_n > t\}}] = e^{-t \lambda_n} .
$$
The second inequality is Markov's inequality.
\end{proof}
Our objective is to find a sharp upper bound for $\sup_{x \in B_n} \EE_x^\tau[T_n]$. As noted in inequality (\ref{lambdalambda0}), finding a lower bound for $\lambda_n^\circ$ is sufficient. Therefore, we assume in this section that $\mathbf{a=0}$.

We introduce the embedded discrete time random walk $(Y_n)_{n \in \N}$, and the jump instants $(J_n)_{n \in \N}$, so that
$$
J_n \le t < J_{n+1} \quad \Rightarrow \quad X_t = Y_n .
$$
Recalling that we assumed here $a=0$, it is clear that conditionally on $Y_n = x$, the time $J_{n+1}-J_{n}$ spent by the walk at site $x$ is an exponential variable of mean $\tau_x$. Let $G_n(x,y)$ be the number of visits before exiting $B_n$ at site $y$ for the walk $Y$ starting at $x$~: 
$$
\hat{T}_n = \inf\{k : Y_k \notin B_n \} \quad \text{ and } \quad G_n(x,y) = \EE_x^\tau \left[ \sum_{k=0}^{\hat{T}_n-1} \1_{\{Y_k = y\}} \right].
$$
Note that $G_n(x,y)$, as the expectation of a functional of $Y$, is non-random.
As a consequence of the above remark, the expected total time spent by the walk $X$ at site $x$ before exiting $B_n$ is $\tau_x$ times the number of visits of $Y$ at site $x$. In other words~:
\begin{equation}
\label{TleG}
\EE_x^\tau[T_n] = \sum_{y \in B_n} G_n(x,y) \tau_y.
\end{equation}
Roughly speaking, we will see that the expectation of this sum behaves like $n^2$ (assuming $\alpha > 1$), and that the probability to be far from the expectation by $n^{d/\alpha}$ is of order $n^{-d}$. To estimate theses fluctuations, our method will be to compute moments after a truncation and centring of the $\tau_x$.
%
%
%
%
%
%
%
To do so, the first thing we need is to find convenient upper bounds for $G_n(\cdot,\cdot)$.

\begin{prop}
\label{Gn}
\begin{enumerate}
\item
\label{Gn2}
There exists $C_1 > 0$ such that for any integer $n$~:
\begin{equation*}
\sum_{y \in B_n} G_n(0,y) \le C_1 n^2.
\end{equation*}
\item
\label{green}
If $d \ge 3$, then there exists $C_2 > 0$ such that for any integer $n$ and any $x \in \Z^d$~:
$$
G_n(0,x) \le \frac{C_2}{(1+\|x\|)^{d-2}}.
$$
\item
\label{Gn1}
If $d = 2$, then there exists $C_3 > 0$ such that for any integer $n$ and any $x \in \Z^d$~:
\begin{equation*}
G_n(0,x) \le C_3 \ln(n).
\end{equation*}
\end{enumerate}
\end{prop}
\begin{proof}
For the first part, note that 
$$
\sum_{y \in B_n} G_n(0,y) = \EE_0^\tau \left[ \sum_{k=0}^{\hat{T}_n-1} \1_{\{Y_k \in B_n\}} \right] = \EE_0^\tau[\hat{T}_n].
$$
As given for instance by \cite[Section XIV.3]{fel1}), the expectation of the exit time of the first coordinate of $Y$ from $\{-n,\ldots,n\}$ is bounded by a constant times $n^2$. It is clear that this quantity is an upper bound for $\EE_0^\tau[\hat{T}_n]$.
The second inequality is a consequence of \cite[Theorem 1.5.4]{law}, while the last comes from \cite[Theorem 1.6.6]{law}. 
\end{proof}
We begin by cutting and centring the random variables $(\tau_x)$. Let $\alpha' < \alpha$ (remember that $\E[\tau_0^{\alpha'}]$ is finite). For technical reasons we impose on $\alpha'$ the additional condition
\begin{equation}
\label{restrictalpha'}
d \le 3 \quad \Rightarrow \quad \alpha' \le 2.
\end{equation}
As we will see in the proof of Theorem~\ref{green3}, this restriction is of no consequence for our purpose.
We define the following truncation of $\tau_x$~:
\begin{equation*}
\td{\tau}_{x,n} = \left|
 \begin{array}{ll}
\tau_x & \text{if } \tau_x \le n^{d/\alpha'}, \\
0 & \text{otherwise}
\end{array}
\right.
\end{equation*}
(observe that with high probability, we have $\tau_x = \td{\tau}_{x,n}$ for every $x \in B_n$), and let $\ov{\tau}_{x,n} = \td{\tau}_{x,n} - \E[\td{\tau}_{x,n}]$. 

We proceed to show the following proposition, that roughly speaking states that fluctuations of order $n^{d/\alpha'}$ of the exit time from $0$ occur with probability smaller than $n^{-d}$.

\begin{prop}
\label{controlfluct}
For any $\beta > d/\alpha'$, there exist $\delta, C > 0$ such that for all $n$~:
\begin{equation*}
\P\left[ \left| \sum_{x \in B_n} G_n(0,x) \ov{\tau}_{x,n} \right| > n^\beta \right] \le \frac{C}{n^{d+\delta}} .
\end{equation*}
\end{prop}
\begin{proof}
Let $m$ be an integer. We have~:
\begin{equation}
\label{moments}
\begin{split}
& \E\left[\left( \sum_{x \in B_n} G_n(0,x) \ov{\tau}_{x,n}  \right)^{2m}\right]  \\
& \qquad =  \sum_{x_1,\ldots, x_{2m}} G_n(0,x_1) \cdots G_n(0,x_{2m}) \E[\ov{\tau}_{x_1,n} \cdots \ov{\tau}_{x_{2m},n} ] \\
& \qquad = \sum_{k=1}^m \sum_{\substack{e_1+\cdots+e_k = 2m \\ e_i \ge 2}} C_{e_1, \ldots, e_k} \sum_{\substack{y_1,\ldots, y_k \\ y_i \neq y_j}} \prod_{i=1}^k G_n(0,y_i)^{e_i} \E[\ov{\tau}_{y_i,n}^{e_i}] \\
& \qquad  \le C(m) \sum_{k=1}^m \sum_{\substack{e_1+\cdots+e_k = 2m \\ e_i \ge 2}} \underbrace{\prod_{i=1}^k \sum_{x \in B_n} G_n(0,x)^{e_i} |\E[\ov{\tau}^{e_i}_{0,n}]|}_{=: \Pi_{e_1,\ldots,e_k}^n} ,
\end{split}
\end{equation}
where, to get the second equality, we chose to decompose $x_1, \ldots, x_{2m}$ the following way~: let $k$ be the cardinal of $\{x_1,\ldots,x_{2m}\}$. We have $\{x_1,\ldots,x_{2m}\} = \{y_1, \ldots, y_k\}$. Then $e_i$ represents then number of occurrences of $y_i$ in $x_1,\ldots,x_{2m}$. We then use the fact that the random variables $(\ov{\tau}_{x,n})_{x \in \Z^d}$ are independent to split the expectation in product form. Note that as $\ov{\tau}_{x,n}$ is a centred random variable, the cases when $e_i=1$ for some $i$ do not contribute to the sum, so it is enough to consider cases when $e_i \ge 2$ (and this implies $k \le m$). It is a nice combinatorics exercise to check that $C_{e_1,\ldots,e_k}$ is the multinomial coefficient associated with $(e_1,\ldots,e_k)$ divided by $k !$, but the important fact is that this term does not depend on $n$.


We will now determine the asymptotic behaviour of the $\Pi_{e_1,\ldots,e_k}^n$. If $d \ge 3$, using part~(\ref{green}) of Proposition~\ref{Gn}, one knows that 
$$
\sum_{x \in B_n} G_n(0,x)^{e_i} \le C \sum_{x \in B_n} {(1+\|x\|)^{- e_i(d-2)}},
$$ 
which, by comparison with an integral, is bounded by~:
$$
\left|
\begin{array}{ll}
C \ln(n) & \text{if } d \ge 4 \text{ or } e_i \ge 3  ,\\
C n & \text{if } d \ge 3.
\end{array}
\right.
$$
On the other hand, $|\E[\ov{\tau}_{0,n}^{e_i}]|$ is bounded when $n$ goes to infinity if $e_i \le \alpha'$, and otherwise 
\begin{equation}
\label{momtau}
|\E[\ov{\tau}_{0,n}^{e_i}]| \le \E[|\ov{\tau}_{0,n}|^{(e_i - \alpha') + \alpha'}] \le  (n^{d/\alpha'})^{e_i-\alpha'} \E[|\ov{\tau}_{0,n}|^{\alpha'}] \le C n^{e_i d / \alpha' - d}.
\end{equation}
We first treat the case $d \ge 4$. We choose $m$ as the smallest integer larger than (or equal to) $\alpha'/2$. All the $\Pi_{e_1,\ldots,e_k}^n$ are bounded by $C \ln(n)^m$ when $n$ goes to infinity except~:
$$
\Pi_{2m}^n \le C \ln(n) n^{2 m d / \alpha' - d }.
$$
It comes, using Markov's inequality, that there exists $C$ such that for any $n$~:
$$
\P\left[\left|\sum_{x \in B_n} G_n(0,x) \ov{\tau}_{x,n} \right| > n^\beta\right] \le C  n^{-d} \ln(n)^m n^{2 m (d / \alpha' - \beta)},
$$
which proves the desired result. 

We go on with the case when $d = 3$ and $\alpha' \le 2$ (see (\ref{restrictalpha'})). We choose $m=2$ in (\ref{moments}) and get~:
$$
\Pi_{2,2}^n \le C n^2 n^{12/\alpha'-6} \quad \text{ and } \quad \Pi_4^n \le C \ln(n) n^{12/\alpha' - 3},
$$
and it comes that~:
$$
\P\left[\left|\sum_{x \in B_n} G_n(0,x) \ov{\tau}_x \right| > n^\beta\right] \le C n^{-3} \ln(n)  n^{4 (3 / \alpha' - \beta)},
$$
which proves the proposition, and we are left with the two-dimensional case. From the estimates of Proposition~\ref{Gn}, we know that
$$
\sum_{x \in B_n} G_n(0,x)^{e_i} \le (C_3 \ln(n))^{e_i-1} \sum_{x \in B_n} G_n(0,x) \le C \ln(n)^{e_i} n^2,
$$
from which we obtain that, provided $e_1+\cdots+e_k = 2m$~:
$$
\Pi_{e_1,\ldots,e_k}^n \le C \ln(n)^{2m} n^{2k} \prod_{i=1}^k |\E[\ov{\tau}^{e_i}_{0,n}]|.
$$
Recalling that (from equation (\ref{momtau}) and the fact that $\alpha' \le 2$),
$$
|\E[\ov{\tau}_{0,n}^{e_i}]| \le C n^{2 e_i / \alpha' - 2},
$$
we obtain, for any sequence $e_1,\ldots,e_k$ such that $e_1+\cdots+e_k = 2m$~:
$$
\Pi_{e_1,\ldots,e_k}^n \le C  \ln(n)^{2m} n^{4m/\alpha'}.
$$
Now we choose $m$ large enough so that~:
$$
\left(\frac{4}{\alpha'} - 2 \beta \right)  m < -2
$$
and apply Markov's inequality.
\end{proof}
The next step is to lift this estimate to the sum of $G_n(0,x) \td{\tau}_{x,n}$.
%
%
%
%
%
\begin{prop}
\label{lift}
There exists $M$ such that for any $\beta > d/\alpha'$, there exist $\delta, C > 0$ such that for all $n$~:
$$
\P\left[ \sum_{x \in B_n} G_n(0,x) \td{\tau}_{x,n} > M n^2 + n^\beta  \right] \le \frac{C}{n^{d+\delta}}.
$$
\end{prop}
\begin{proof}
Note that as $\E[\td{\tau}_{x,n}] \le \E[\tau_0]$, and using part~(\ref{Gn2}) of Proposition~\ref{Gn}~:
$$
\sum_{x \in B_n} G_n(0,x) \E[\td{\tau}_{x,n}] \le C_1 \E[\tau_0] n^2.
$$
It comes that
\begin{equation*}
\P\left[ \sum_{x \in B_n} G_n(0,x) \td{\tau}_{x,n} > C_1 \E[\tau_0] n^2 + n^\beta   \right] \le \P\left[ \sum_{x \in B_n} G_n(0,x) \ov{\tau}_{x,n} >  n^\beta  \right]  ,
\end{equation*}
on which we apply Proposition \ref{controlfluct}. 
\end{proof}
%
%
%
%
%
We can now carry this result back to $\sup_{x \in B_n} \EE_x^\tau[T_n]$.
\begin{prop}
\label{tpssortie}
There exists $M'$ such that for any $\beta > d/\alpha'$, almost every environment and $n$ large enough~:
$$
\sup_{x \in B_n} \EE_x^\tau [T_n] \le n^\beta + M' n^2.
$$
\end{prop}
\begin{proof}
We first need to relate $\EE_x^\tau[T_n]$ with the estimates proved before (which concern only $\EE_0^\tau[T_n]$). Let $T_n^x$ be the exit time from $x+B_n$. Since for any $x \in B_n$, we have $B_n \subseteq x+B_{2n}$, it comes that almost surely $T_n \le T_{2n}^x$, so $\EE_x^\tau [T_n] \le \EE_x^\tau [T_{2n}^x]$, the latter having same law as $\EE_0^\tau [T_{2n}]$ under $\P$.

Let $M'>0$ and let $i$ be an integer. We consider~:
\begin{equation}
\label{decompose}
\P\left[ \sup_{n \ge 2^i} \frac{\sup_{x \in B_n} \EE_x^\tau [T_n]}{n^\beta + M' n^2} > 1  \right] 
\le \sum_{j=i}^\infty \P\left[ \sup_{2^j \le n < 2^{j+1}} \frac{\sup_{x \in B_n} \EE_x^\tau [T_{2n}^x]}{n^\beta + M' n^2} > 1  \right]  .
\end{equation}
We bound the general term of this series by 
$$
\P\left[\sup_{x \in B_{2^{j+1}}} \EE_x^\tau [T_{2^{j+2}}^x] > 2^{j\beta} + M' 2^{2j} \right] ,
$$
which we bound by $A_j+|B_{2^{j+1}}| A'_j$, where~:
\begin{equation}
\label{Aj}
A_j = \P\left[\exists x \in B_{2^{j+2}} : \tau_x > 2^{(j+2) d/\alpha'} \right],
\end{equation}
$$
A'_j =  \P\left[ \sum_{x \in B_{2^{j+2}}} G_n(0,x) \td{\tau}_{x,2^{j+2}} > 2^{j\beta} + M' 2^{2j} \right] .
$$
We first estimate $A_j$. Take $\alpha''$ such that $\alpha' < \alpha'' < \alpha$. It comes from assumption~1' (see (\ref{ppF})) that for all $y$ large enough~:
$$
\P[\tau_0 > y] \le y^{-\alpha''}.
$$
One gets that for $j$ large enough~:
$$
A_j \le 1 - \left( 1- 2^{-j d\alpha''/\alpha'}\right)^{|B_{2^{j+2}}|} = 1- \exp\left(|B_{2^{j+2}}| 2^{-j d\alpha''/\alpha'} (1+o(1))\right),
$$
which is the general term of a convergent series. 

Now for $A'_j$, using Proposition \ref{lift}, we see that choosing $M' = 16M$, the term $|B_{2^{j+1}}|A'_j$ is bounded by $C 2^{-j\delta}$ for some $\delta > 0$. Therefore, the series in the right-hand side of \ref{decompose} converges (and tends to $0$ when $i$ goes to infinity), which proves the proposition.
\end{proof}
%
%
%
%
%
%
We can now conclude~:
\begin{thm}
\label{green3}
\begin{enumerate}
\item 
\label{green3a}
If $d/\alpha \ge 2$, then for almost every environment~:
$$
\limsup_{n \to \infty} - \frac{\ln(\lambda_n)}{\ln(n)} \le \frac{d}{\alpha}.
$$
\item 
\label{supn2}
If $d/\alpha < 2$, then there exists $C$ such that for almost every environment and all $n$ large enough~:
$$
\lambda_n \ge \frac{C}{n^2}.
$$
\end{enumerate}
\end{thm}
\begin{proof}
Due to inequality~(\ref{lambdalambda0}), it is enough to show these results for $\lambda_n^\circ$. If $d \ge 4$ or $\alpha \le 2$, it is a consequence of Proposition \ref{tpssortie} together with Proposition \ref{sortie} (making $\alpha'$ tend to $\alpha$). Now if $d \in \{2, 3\}$ and $\alpha > 2$, then we can choose $\alpha' = 2$, in which case $d/2 < 2$, and part (\ref{supn2}) of the theorem still holds.
\end{proof}
%
%
%
%
%
%
\section{Upper bounds on $\lambda_n$}
\label{s:upper}
\setcounter{equation}{0}
We now give upper bounds on $\lambda_n$. Our method is clear from equation (\ref{variational}), that we recall here~:
$$
\lambda_n = \inf_{\substack{f \in L^2(B_n) \\ f \neq 0}} \frac{\cE(f,f)}{(f,f)}.
$$
Picking a function in $L^2(B_n)$ gives an upper bound, and the problem is to choose the function well enough (i.e. looking more or less like the eigenfunction) to get a sharp bound.

\subsection{The one-dimensional case}
\begin{thm}
\label{upper1d}
We assume $d=1$. There exists $C > 0$ such that for almost every environment and all $n$ large enough~:
$$
\lambda_n \le \frac{C}{n \ \sum_{x \in B_{n/4}} \tau_x}.
$$
\end{thm}
\begin{proof}
For $a = 0$, a ``triangle function'' that takes the value $0$ on $-(n+1)$ and $(n+1)$, the value $1$ on $0$ and is piecewise linear would do well. But for general $a$, this function is not appropriate, and we will construct instead a function that looks like it, but is constant around deep traps.

Let $M > 0$ be such that $\P[\tau_0 > M] \le 1/8$. Because of the law of large numbers, one gets~:
$$
\frac{1}{n} \ | \{ k \in \{-n-1,\ldots, 0\} \ : \ \tau_k > M \} | \xrightarrow[n \to \infty]{a.s.} \frac{1}{8}.
$$
Almost surely, for $n$ large enough, the two following conditions are satisfied~:
\begin{equation}
\label{gauche}
| \{ k \in \{-n-1,\ldots, 0\} \ : \ \tau_k > M \} | \le \frac{n}{4},
\end{equation}
\begin{equation}
\label{droite}
| \{ k \in \{0 ,\ldots, n+1\} \ : \ \tau_k > M \} | \le \frac{n}{4}.
\end{equation}
Let us first construct the left part of our function~: let $l : -\N \to \R$ be such that $l(k) = 0$ for all $k < -n$, and for all $k \in \{-n, \ldots, 0\} $~:
$$
l(k)-l(k-1) = 
\left|
\begin{array}{ll}
0 & \text{if } \tau_{k-1} > M \text{ or } \tau_k > M, \\
1/n & \text{otherwise}.
\end{array}
\right.
$$
The function $l$ is made in such a way that for all $k$ for which it makes sense~:
\begin{equation}
\label{var}
\tau_k^a \tau_{k+1}^a (l(k+1)-l(k))^2 \le \frac{M^{2a}}{n^2}.
\end{equation}
Moreover, when (\ref{gauche}) is satisfied, there are at most half of the edges on which the function is constant, so $l(0) \ge 1/2$. In this case, and as for any $k$ we have $l(k) - l(k-1) \le 1/n$, it comes that $l(k) \ge 1/4$ when $k \ge -n/4$.

We define in the same way a right part $r : \N \to \R$ such that $r(k) = 0$ for all $k > n$, and for all $k \in \{n, \ldots, 0\} $~:
$$
r(k)-r(k+1) = 
\left|
\begin{array}{ll}
0 & \text{if } \tau_{k} > M \text{ or } \tau_{k+1} > M, \\
1/n & \text{otherwise}.
\end{array}
\right.
$$
The function $r$ satisfies the same small variation property as in (\ref{var}). Similarly, when (\ref{droite}) is satisfied, we have that $r(0) \ge 1/2$ and $r(k) \ge 1/4$ for all $k \le n/4$.

Now we connect the two parts $l$ and $r$ preserving this small variation property. Let $m = \min(l(0),r(0))$. We define $f : \Z \to \R$ by
$$
f(x) =
\left|
\begin{array}{ll}
\min(l(x),m) & \text{if } x < 0 ,\\
\min(r(x),m) & \text{otherwise}.
\end{array}
\right.
$$
We have therefore~:
$$
\cE(f, f) \le \frac{2 M^{2a}}{n}.
$$
On the other hand, for $n$ large enough, (\ref{gauche}) and (\ref{droite}) are satisfied, and in this case $m \ge 1/2$ and $f(k) \ge 1/4$ for all $k$ such that $-n/4 \le k \le n/4$. Thus~:
$$
(f,f) \ge \frac{1}{16} \sum_{-n/4 \le k \le n/4} \tau_k,
$$
and we finally obtain, for all $n$ large enough~:
$$
\lambda_n \le \frac{\cE(f,f)}{(f,f)} \le \frac{32 M^{2a}}{n \ \sum_{x \in B_{n/4}} \tau_x}.
$$
\end{proof}

\subsection{Large dimension, anomalous behaviour}
The results proved in this part are in fact valid in any dimension and for any $\alpha > 0$, but they are sharp only in the regime given in the title, that is for $d \ge 2$ and $2 \alpha \le d$.
\begin{thm}
\label{relax}
\begin{enumerate}
\item
\label{relaxprob}
For any $\eps > 0$, there exists $M > 0$ such that for all $n$ large enough~:
$$
\P\left[\lambda_n \max_{B_{n-1}} \tau \le M\right] \ge 1 - \eps.
$$
\item
\label{relaxps}
For any $\eps > 0$ and almost every environment~:
$$
n^{d/\alpha-\eps} \lambda_n \xrightarrow[n \to \infty]{} 0.
$$
\end{enumerate}
\end{thm}
\begin{proof}
Let $K$ be the set of first and second neighbours of $0$, namely $K = \{x \in \Z^d : 1 \le \|x\| \le 2 \}$, and $c$ the number of edges from a point of $\{x : \|x\| = 1 \}$ to a point of $\{x : \|x\| = 2 \}$.  Write $M_x = \max_{x+K} \tau$. If we choose the function that takes value $1$ on site $x \in B_{n-1}$ and its neighbours, and $0$ elsewhere, namely~:
$$
f(z) = 
\left|
\begin{array}{ll}
1 & \text{if } \|z-x\| \le 1 ,\\
0 & \text{otherwise},
\end{array}
\right.
$$
then we see that for any $x \in B_{n-1}$~:
\begin{equation}
\label{supgene}
\lambda_n \le \frac{c(M_x)^{2a}}{\tau_x}.
\end{equation}
Let $x_n \in B_{n-1}$ be such that $\tau_{x_n} = \max_{B_{n-1}} \tau$. We have~:
$$
\lambda_n \le \frac{c(M_{x_n})^{2a}}{\max_{B_{n-1}} \tau}.
$$
So we get~:
$$
\P\left[\lambda_n \max_{B_{n-1}} \tau \ge M\right] \le \P\left[c (M_{x_n})^{2a} \ge M\right].
$$
Now recall that $M_{x_n}$ is the maximum over all neighbours and second neighbours of $x_n$, so it should look like taking the maximum over all neighbours and second neighbours of, say, $0$. More precisely, conditionally on $\max_{B_{n-1}} \tau = \tau_z$ for some fixed $z$, the law of $(\tau_x)_{x \in B_{n-1} \setminus \{z\}}$ is invariant under permutation. Therefore, provided $z \in B_{n-2} \setminus K$ and conditionally on $\max_{B_{n-1}} \tau = \tau_z$, the random variables $M_z$ and $M_0$ have the same law. Summing over all $z \in B_{n-2} \setminus K$, we get that  conditionally on the event $E_n$ that $x_n \in B_{n-2} \setminus K$, the random variables $M_0$ and $M_{x_n}$ have the same law. We obtain~:
$$
\P\left[c (M_{x_n})^{2a} \ge M\right] \le \P\left[c (M_{0})^{2a} \ge M\right] + \P\left[E_n^c\right].
$$
The law of $x_n$ being uniform in $B_{n-1}$, we have that $\P\left[E_n^c\right]$ goes to $0$ when $n$ goes to infinity. First part of the theorem comes choosing $M$ large enough.

We now turn to the second assertion of the proposition. Defining~:
$$
\ov{M}_n = \max_{x \in B_{n-1}} \frac{\tau_x}{(M_x)^{2a}},
$$
we will show that for any $\eps > 0$~:
\begin{equation}
\label{infty}
\frac{\ov{M}_n}{n^{d/\alpha-\eps}}   \xrightarrow[n \to \infty]{\text{a.s.}} + \infty,
\end{equation}
which will prove the result via equation (\ref{supgene}). There exists $k > 0$ such that $\P[(M_x)^{2a} > k] < 1/2$. Thus (note that $M_x$ and $\tau_x$ are independent)~:
$$
\P\left[\frac{\tau_x}{(M_x)^{2a}} \ge y\right] \ge \frac{\P[\tau_x \ge ky]}{2}  = \frac{F(ky)}{2}.
$$
Hence, for all $K > 0$~:
$$
\P[\ov{M}_n \le n^{d/\alpha-\eps} K] \le \left(1-\frac{F(kKn^{d/\alpha-\eps})}{2}\right)^{(2n-1)^d},
$$
and recalling that, as a consequence of assumption 1' (see (\ref{ppF})), for all $\beta < \alpha$, $F(y) \le y^{-\beta}$ for all $y$ large enough, one can see that the term on the right-hand side of the former equality is the general term of a convergent series, and thus apply the Borel-Cantelli lemma.
\end{proof}

\subsection{Regular behaviour}
In what follows our assumption will be that $\E[\tau_0^a]$ is finite. In particular, all results will be valid under the condition that $\E[\tau_0]$ is finite (or if $a=0$).

We write $(e_i)_{1 \le i \le d}$ for the canonical base of $\R^d$.
\begin{prop}
\label{lln}
Let $f : [-1,1]^d \to \R$ be a continuous function. If $\E[\tau_0^a]$ is finite, then for all $i \in \{1,\ldots, d\}$~:
\begin{equation}
\label{llneq}
\frac{1}{(2n+1)^d} \sum_{x \in B_n} \tau_x^a \tau_{x+e_i}^a \ f(x/n) \xrightarrow[n \to \infty]{\text{a.s.}} \E[\tau_0^a]^2 \int_{[-1,1]^d} f(x) \d x.
\end{equation}
\end{prop}
\begin{proof}
If $f$ is piecewise constant, then the limit (\ref{llneq}) is proved by separating the sum over $B_n$ into two parts $B_n'$ and $B_n''$ so that $(\tau_x^a \tau_{x+e_i}^a)_{x \in B_n'}$ and $(\tau_x^a \tau_{x+e_i}^a)_{x \in B_n''}$ are two families of independent random variables, and then applying the law of large numbers. For a continuous $f$, one can approximate uniformly $f$ by piecewise constant functions from above and below, and the result follows.
\end{proof}
For all $f : [-1,1]^d \to \R$ and all integer $n$, we define the function $f_n : \Z^d \to \R$ by $f_n(x) = f(x/n)$ if $x \in B_n$, and $f_n(x) = 0$ otherwise. Note that $f_n \in L^2(B_n)$.
\begin{prop}
\label{lln+}
Let $f : [-1,1]^d \to \R$ be a twice continuously differentiable function that takes value $0$ on the boundary of $[-1,1]^d$. If $\E[\tau_0^a]$ is finite, then~:
$$
\frac{n^2}{(2n)^d} \cE(f_n, f_n) \xrightarrow[n \to \infty]{\text{a.s.}} \E[\tau_0^a]^2 \int_{[-1,1]^d} \|\nabla f(x)\|_2^2 \d x.
$$
\end{prop}
Recall the following equality~:
$$
\cE(f_n, f_n) = \sum_{i = 1}^d \sum_{x \in B_n} \tau_x^a \tau_{x+e_i}^a \left(f\left(\frac{x}{n}\right) - f\left(\frac{x+e_i}{n}\right)\right)^2.
$$
As we assumed $f$ to be twice continuously differentiable, it comes that for all $\eps > 0$ and $n$ large enough~:
$$
\forall x \in B_n  : x+e_i \in B_n \Rightarrow \left|\left(f\left(\frac{x}{n}\right) - f\left(\frac{x+e_i}{n}\right)\right)^2 - \frac{1}{n^2} {\frac{\partial f}{\partial x_i}\left(\frac{x}{n}\right)}^2 \right| \le \frac{\eps}{n^2},
$$
and note that if $x \in B_n$ and $x+e_i \notin B_n$, then $f(x/n) = f((x+e_i)/n) = 0$, so this case does not contribute to the sum. The result follows using the previous proposition.
\begin{thm}
\label{taua}
If $\E[\tau_0^a]$ is finite, then there exists $C$ such that almost surely, for all $n$ large enough~:
$$
\lambda_n \le \frac{C}{n^2} \frac{n^d}{\sum_{x \in B_{n/2}} \tau_x}.
$$
\end{thm}
\begin{proof}
Taking $f(x) = \prod_{i=1}^d \sin\left(\frac{\pi x_i}{2}\right)$ in Proposition \ref{lln+}, we get that for almost every environment~:
$$
\cE(f_n, f_n) \sim \frac{d \pi^2}{4} \frac{(2n)^d}{n^2} \E[\tau_0^a]^2  \qquad (n \to +\infty).
$$
On the other hand, if $x \in B_{n/2}$, then $f(x) \ge 2^{-d/2}$, thus~:
$$
(f_n,f_n) \ge 2^{-d/2} \sum_{x \in B_{n/2}} \tau_x,
$$
therefore the proposition holds for any $C > 2^{3d/2-2} d \pi^2 \E[\tau_0^a]^2$.
\end{proof}
%
%
%
%
%
%
\section{The distinguished path method}
\label{s:cs}
\setcounter{equation}{0}
We present here a more direct method to get a lower bound on $\lambda_n$ (close to the one presented e.g. in \cite[Theorem 3.2.3]{SC}, but adapted to treat the case of Dirichlet boundary condition), and show that it does not provide a sharp estimate when $d \ge 2$. Note that in dimension one, \cite[Section 3.7]{chen} proves that this technique is always sharp, and one can verify that it gives indeed the expected lower bound. This method also proved efficient in larger dimension in \cite[Section 3]{fm} in the context of random walks among random conductances.

For all $x \in B_n$, we give ourselves a path $\gamma_n(x)$ from some point of $\partial B_n$ to $x$ (that apart from the starting point, visits only points in $B_n$). Let $\gamma_n(x) = (x^0, \ldots, x^l)$. For an edge $e$, we note $e \in \gamma_n(x)$ if $e = (x^{i},x^{i+1})$ for some $i$, and in this case, we write $\d f(e) = f(x^{i+1})-f(x^i)$, and $\mathfrak{Q}(e) = \tau_{x^i}^a \tau_{x^{i+1}}^a$. Let $E_n$ be the set of edges that go from a point of $B_n$ to a point of $B_n \cup \partial B_n$. We give ourselves a weight function $W_n : E_n \to (0,+\infty)$. We define the $W_n$-length of a path $\gamma$ as~:
$$
l_n(\gamma) = \sum_{e \in \gamma} \frac{1}{W_n(e)}.
$$
Note that, as we assumed that $\tau \ge 1$, we have that $\mathfrak{Q}(e) \ge 1$ (and there is equality when $a=0$). Using Cauchy-Schwarz inequality, we get~:
\begin{eqnarray*}
f(x)^2 & = & \left( \sum_{e \in \gamma_n(x)} \d f(e) \right)^2 \\
& \le & \sum_{e \in \gamma_n(x)} \frac{1}{W_n(e)\mathfrak{Q}(e)} \sum_{e \in \gamma_n(x)} \d f(e)^2 W_n(e) \mathfrak{Q}(e) \\
& \le & l_n(\gamma_n(x)) \sum_{e \in \gamma_n(x)} \d f(e)^2 W_n(e) \mathfrak{Q}(e)
\end{eqnarray*}
\begin{eqnarray*}
\sum_{x \in B_n} f(x)^2 \tau_x & \le & \sum_{x \in B_n} l_n(\gamma_n(x)) \tau_x \sum_{e \in \gamma_n(x)} \d f(e)^2 W_n(e) \mathfrak{Q}(e) \\
& \le & \sum_{e \in E_n} \d f(e)^2 \mathfrak{Q}(e) W_n(e) \sum_{x : e \in \gamma_n(x)} l_n(\gamma_n(x)) \tau_x.
\end{eqnarray*}
Note that 
$$
\cE(f,f) = \sum_{e \in E_n} \d f(e)^2 \mathfrak{Q}(e),
$$
so letting 
$$
\mathcal{M}_n := \max_{e \in E_n} W_n(e) \sum_{x : e \in \gamma_n(x)} l_n(\gamma_n(x)) \tau_x,
$$
we obtain the following lower bound on $\lambda_n$ (similar to \cite[Theorem 3.2.3]{SC})~:
$$
\lambda_n \ge \frac{1}{\mathcal{M}_n}.
$$
Let us see that, however $W_n$ and $\gamma_n(x)$ are chosen, it cannot lead to a sharp bound if $d \ge 2$ and $\alpha < d$. Let $z \in B_{n/2}$ be such that $\tau_z$ is maximal. The site $z$ is such that $\tau_z \simeq n^{d/\alpha}$ and $|\gamma_n(z)| \ge n/2$. Now choose $e \in \gamma_n(z)$ so that $W_n(e)$ is maximal. We have~:
$$
\mathcal{M}_n \ge \sum_{e' \in \gamma_n(z)} \frac{W_n(e)}{W_n(e')} \tau_z \ge |\gamma_n(z)| \tau_z \gtrsim n^{1+d/\alpha},
$$
where we would have hoped to find $n^{\max(2,d/\alpha)}$. So this method cannot give the appropriate exponent if $\alpha < d$.

Still, note that if one chooses $W_n$ constant equal to $1$, and the shortest paths for $(\gamma_n(x))_{x \in B_n}$, one can show using results of \cite{BK} that $\mathcal{M}_n$ is indeed of order $n^{\max(2,1+d/\alpha)}$, which gives an alternative proof of a lower bound for the principal eigenvalue when $\alpha \ge d$.

\noindent \textbf{Acknowledgments.} The author would like to thank his Ph.D. advisors, Pierre Mathieu and Alejandro Ram\'irez, for many insightful discussions about this work as well as detailed comments on earlier drafts, and G\'erard Ben Arous for suggesting this problem.


\begin{thebibliography}{99}

\bibitem[Al81]{alex} S.~Alexander. Anomalous transport properties for random-hopping and random-trapping models. \emph{Phys. Rev. B} \textbf{23} (6), 2951-2955 (1981)

\bibitem[BK65]{BK} L.E.~Baum, M.~Katz. Convergence rates in the law of large numbers. \emph{Trans. Amer. Math. Soc.} \textbf{120} (1) 108-123 (1965).

\bibitem[B\v{C}05]{dim1} G.~Ben Arous, J.~\v{C}ern\'y. Bouchaud's model exhibits two different aging regimes in dimension one. \emph{Ann. Appl. Probab.} \textbf{15} (2), 1161-1192 (2005).

\bibitem[B\v{C}06]{rev} G.~Ben Arous, J.~\v{C}ern\'y. Dynamics of trap models. \emph{Les Houches summer school lecture notes}, Elsevier (2006).

\bibitem[B\v{C}07]{dim2} G.~Ben Arous, J.~\v{C}ern\'y. Scaling limit for trap models on $\Z^d$. \emph{Ann. Probab.} \textbf{35} (6) 2356-2384 (2007).


\bibitem[BGT]{reg} N.H.~Bingham, C.M.~Goldie, J.L.~Teugels. \emph{Regular variation}. Cambridge University Press (1989).

\bibitem[BF05]{bovfag1} A.~Bovier, A.~Faggionato. Spectral characterization of aging: the REM-like trap model. \emph{Ann. Appl. Probab.} \textbf{15} (3), 1997-2037 (2005).

\bibitem[BF08]{bovfag2} A.~Bovier, A.~Faggionato. Spectral analysis of Sinai's walk for small eigenvalues. \emph{Ann. Probab.} \textbf{36} (1), 198-254 (2008).

\bibitem[Bo92]{bou} J.-P.~Bouchaud. Weak ergodicity breaking and aging in disordered systems. \emph{J. Phys. I} (France) \textbf{2}, 1705-1713 (1992).




\bibitem[Chen]{chen} M.-F.~Chen. \emph{Eigenvalues, inequalities, and ergodic theory}. Springer (2005).




\bibitem[DFGW89]{masi} A.~De Masi, P.A.~Ferrari, S.~Goldstein, W.D.~Wick. An invariance principle for reversible Markov processes. Applications to random motions in random environments. \emph{J. Statistical Physics} \textbf{55} (3-4), 787-855 (1989).


\bibitem[Fel1]{fel1} W.~Feller. \emph{An introduction to probability theory and its applications}, vol. I, third edition. John Wiley \& Sons, Inc. (1968).

\bibitem[Fel2]{fel} W.~Feller. \emph{An introduction to probability theory and its applications}, vol. II, second edition. John Wiley \& Sons, Inc. (1971).

\bibitem[FIN02]{fin} L.R.G.~Fontes, M.~Isopi, C.M.~Newman. Random walks with strongly inhomogeneous rates and singular diffusions: convergence, localization and aging in one dimension. \emph{Ann. Probab.} \textbf{30} (2), 579-604 (2002).

\bibitem[FM06]{fm} L.R.G.~Fontes, P.~Mathieu. On symmetric random walks with random conductances on $\Z^d$. \emph{Probab. Theory Relat. Fields} \textbf{134}, 565-602 (2006).





\bibitem[Law]{law} G.F.~Lawler. \emph{Intersections of random walks}. Probability and its applications, Birkh\"auser (1991).

\bibitem[LP]{trees} R.~Lyons, with Y.~Peres. \emph{Probability on trees and networks}. Cambridge University Press, in preparation. Current version available at \url{http://mypage.iu.edu/~rdlyons/}.


\bibitem[MB97]{mb97} R.~M\'elin, P.~Butaud. Glauber dynamics and ageing. \emph{J. Physique I} \textbf{7} (5), 691-710 (1997).

\bibitem[Mo08]{version1} J.-C.~Mourrat. Principal eigenvalue for random walk among random traps on $\Z^d$. arXiv:0805.0706v1 (2008).




\bibitem[Pet]{pet} V.V.~Petrov. \emph{Limit Theorems of Probability Theory - Sequences of Independent Random Variables}. Oxford studies in probability (1995).

\bibitem[Res]{res} S.I.~Resnick. \emph{Extreme values, regular variation, and point processes}. Springer-Verlag (1987).

\bibitem[RMB00]{rmb} B.~Rinn, P.~Maass, J.-P.~Bouchaud. Multiple scaling regimes in simple aging models. \emph{Phys. Rev. Lett.} \textbf{84} (23) 5403-5406 (2000).

\bibitem[SC97]{SC} L.~Saloff-Coste. Lectures on finite Markov chains. \emph{Lectures on probability theory and statistics (Saint-Flour 1996)}, Lecture Notes in Math. \textbf{1665}, Springer, 301-413 (1997).




\end{thebibliography}
\end{document}